\documentclass{amsart}

\usepackage[hang,flushmargin]{footmisc}

\usepackage[font=footnotesize,labelfont=footnotesize]{caption, subcaption}

\usepackage{amssymb}
\usepackage{amscd}
\usepackage{amsmath}
\usepackage{amsthm} 
\usepackage[utf8]{inputenc}
\usepackage{amsfonts}
\usepackage{amstext}
\usepackage[foot]{amsaddr}
\usepackage{color}
\usepackage{a4wide}
\usepackage{graphicx}
\usepackage{wrapfig}

\usepackage{mathtools}
\mathtoolsset{showonlyrefs}
\numberwithin{equation}{section}
\usepackage{marginnote}

\usepackage{hyperref}
\usepackage[svgnames]{xcolor}

\theoremstyle{plain}
\newtheorem{theorem}{Theorem}[section]
\newtheorem*{theorem*}{Theorem}

\newtheorem{lemma}[theorem]{Lemma}

\newtheorem{proposition}[theorem]{Proposition}

\newtheorem*{definition*}{Definition}
\newtheorem{corollary}[theorem]{Corollary}

\newcommand{\F}{\mathcal{F}}

\newcommand{\R}{{\mathbb R}}  \newcommand{\Z}{{\mathbb Z}} \newcommand{\N}{{\mathbb N}}

\newcommand{\C}{{\mathbb C}}

\newcommand{\diag}{\mathrm{diag}}
\newcommand{\vol}{\textnormal{vol}}

\usepackage{wasysym}

\definecolor{newgreen}{rgb}{0.518,0.741,0.0}

\allowdisplaybreaks
\title[A note on energy minimization]{A note on energy minimization in dimension 2}
\author[M.~Faulhuber]{Markus Faulhuber}
\author[I.~Shafkulovska]{Irina Shafkulovska}
\author[I.~Zlotnikov]{Ilia Zlotnikov}

\address{NuHAG, Faculty of Mathematics, University of Vienna\newline Oskar-Morgenstern-Platz 1, 1090 Vienna, Austria}
\date{}

\begin{document}

\begin{abstract}
    Proving the \textit{universal optimality of the hexagonal lattice} is one of the big open challenges of nowadays mathematics. We show that the hexagonal lattice outperforms certain ``natural" classes of periodic configurations. Also, we rule out the option that the canonical non-lattice rival -- the honeycomb -- has lower energy than the hexagonal lattice at any scale.
\end{abstract}

\thanks{This work was supported by the Austrian Science Fund (FWF) grant P33217.}
\maketitle

\section{Problems and results}
\subsection{Energy minimization and universal optimality}
For a periodic set $\Gamma$ in $\R^d$ and a potential function $p:\R_+ \to \R$ we consider the quantity 
\begin{equation}
    E_p(\Gamma) = \lim_{R \to \infty} \frac{1}{|\Gamma \cap B_R^d(0)|}\sum_{\substack{\gamma, \gamma' \in \Gamma \cap B_R^d(0) \\ \gamma \neq \gamma'}} p(|\gamma-\gamma'|),
\end{equation}
where $B_R^d(0)$ is the ball of radius $R$ centered at 0 in $\R^d$. The above quantity is the energy of $\Gamma$. The density of $\Gamma$  gives the average number of points per unit volume and is given by
\begin{equation}
    \rho(\Gamma) = \lim_{R \to \infty} \frac{|\Gamma \cap B_R^d(0)|}{\vol(B_R^d(0))}.
\end{equation}
Throughout this work, we will assume that $\rho(\Gamma)$ is fixed. 
As we consider periodic configurations, we can write $\Gamma$ as a finite union of translates of a lattice $\Lambda\subseteq\R^d$;
\begin{equation}\label{eq:periodic_config}
    \Gamma = \bigcup \limits_{j=1}^J \Lambda+x_j
    \quad \text{ with } \quad
    x_j-x_k\notin \Lambda \quad \text{for all } j\neq k.
\end{equation}
A lattice is always assumed to be full-rank. The energy is invariant under shifts;
\begin{equation}
    E_p(\Gamma) = E_p(\Gamma + x), \quad x \in \R^d.
\end{equation}
Due to the special structure of $\Gamma$ we can re-write the energy in the following way
\begin{equation}\label{eq:energy}
    E_p(\Gamma) = \frac{1}{J} \sum\limits_{j,k=1}^J\  \sum\limits_{\lambda\in\Lambda\setminus\{x_j-x_k\} } p(|\lambda+x_j-x_k|).
\end{equation}
The goal is to minimize \eqref{eq:energy} subject to $\Gamma$ (density fixed). In accordance with Cohn and Kumar \cite{CohKum07} (see also \cite{Coh-Via22}) we introduce \textit{universal optimality} in the following way.
\begin{definition*}[Universal optimality]
    Among all periodic configurations $\Gamma \subset \R^d$ of fixed density, we say $\Gamma_\alpha$ is a ground state of \eqref{eq:energy} if for a fixed $\alpha > 0$
    \begin{equation}
        E_{\phi_\alpha}(\Gamma_\alpha) \leq E_{\phi_\alpha}(\Gamma),
        \quad \text{ where } \quad
        \phi_\alpha(r) = e^{-\alpha r^2}, \, r \in \R_+.
    \end{equation}
    A periodic configuration $\Gamma_0$ is universally optimal if it is a ground state for all~$\alpha > 0$;
    \begin{equation}
        E_{\phi_\alpha}(\Gamma_0) \leq E_{\phi_\alpha}(\Gamma), \quad \forall \alpha > 0.
    \end{equation}
\end{definition*}

\subsection{Results}
In this work, we will prove that at all scales the hexagonal lattice $\Lambda_2$ has lower energy than the honeycomb structure. Given the fact that we already know that no other lattice can beat $\Lambda_2$ \cite{Mon88}, the honeycomb structure was the most natural rival. We will also exclude the class of tensor periodic structures having lower energy than $\Lambda_2$ by reducing the dimension of the problem and exploiting universal optimality of the scaled integers \cite{CohKum07}. Also, we pose the problem of distributing 3 points in the unit square (considered periodically) so that they minimize energy.
\begin{theorem}\label{thm:hex_vs_honeycomb}
    For any fixed $\alpha > 0$ and fixed density 1 (and hence any fixed density) the hexagonal lattice $\Lambda_2$ has a smaller energy than the honeycomb structure $\Gamma_\textnormal{\hexagon}$;
    \begin{equation}
        E_{\phi_\alpha}(\Lambda_2) < E_{\phi_\alpha}(\Gamma_\textnormal{\hexagon}), \quad \forall \alpha > 0.
    \end{equation}
\end{theorem}
\begin{theorem}\label{thm:tensor}
    We denote the Cartesian tensor of 1-$d$ periodic configurations by
    \begin{equation}\label{eq:Cartesian_thm}
        \Gamma_\times = \bigtimes\limits_{l=1}^d\Gamma_l,
        \quad \Gamma_l \text{ a periodic configuration in } \R.
    \end{equation}
    Denote by $\Gamma_\otimes$ the minimizing configuration of the energy among periodic configurations of type \eqref{eq:Cartesian_thm}. Then, for any fixed $\alpha > 0$, we have
    $
        E_{\phi_\alpha}(\Gamma_\otimes) \leq E_{\phi_\alpha}(\Gamma_\times)
    $
    with equality if and only if $\Gamma_l$ is equispaced for all $l=1, \ldots, d$. In other words, $\Gamma_\otimes$ is a (cuboid-shaped) lattice. Moreover, for any fixed density $\rho > 0$, the cubic lattice $\rho^{-1/d} \Z^d$ is universally optimal among all periodic configurations of type \eqref{eq:Cartesian_thm}.
\end{theorem}
\begin{corollary}\label{cor:2d}
    In dimension 2, consider the set of lattices, the set of Cartesian tensor configurations and the set of unions of two shifted rectangular or hexagonal lattices. We denote these sets, respectively, by
    \begin{align}
        \mathbf{\Lambda} & = \{\Lambda \subset \R^2 \mid \Lambda \text{ a lattice}\}, \qquad
        \mathbf{\Gamma}_\times = \{ \Gamma \subset \R^2 \mid \Gamma \text{ as in } \eqref{eq:Cartesian_thm} \}\\
        \mathbf{\Gamma}_x & = \{\Gamma \subset \R^2 \mid \Gamma = \Lambda \cup (\Lambda+x), \, \Lambda \in \{(a \Z \times a^{-1} \Z), \, \Lambda_2\}, \, a>0, x \in \R^2 \},
    \end{align}
    We denote the union of these sets by
    \begin{equation}\label{eq:set_Xi}
        \mathbf{\Xi} = \mathbf{\Lambda} \cup \mathbf{\Gamma}_\times \cup \mathbf{\Gamma}_x.
    \end{equation}
    For any fixed $\alpha > 0$ and for any fixed density $\rho > 0$, the hexagonal lattice $\rho^{-1/2} \Lambda_2$ is the unique minimizer among periodic configurations from the set $\mathbf{\Xi}$;
    \begin{equation}
        E_{\phi_\alpha}(\rho^{-1/2} \Lambda_2) \leq E_{\phi_\alpha}(\Gamma), \quad \Gamma \in \mathbf{\Xi}.
    \end{equation}
    In other words, the hexagonal lattice $\Lambda_2$ is universally optimal in $\mathbf{\Xi}$.
\end{corollary}
\noindent
\textbf{Remark}. We also derive a minimality result for the hexagonal lattice in the set
\begin{align}
    \mathbf{\Gamma}_{x,x^*} & = \{ \Gamma \subset \R^2 \mid \Gamma = \Lambda \cup (\Lambda + x) \cup (\Lambda+x^*), \\
    & \hspace{60pt} x \in [0,1]^2, x^* = (1,1)-x, \, \Lambda \in \{\Z^2, \Lambda_2 \} \}. 
\end{align}
For $\Lambda_2$ the result is not too interesting as the optimal configuration is a scaled version of $\Lambda_2$. For $\Z^2$, we prove that the pair $x_0=(1/3,1/3)$ and $x_0^*=(2/3,2/3)$ as well as the pair $\widetilde{x}_0=(1/3,2/3)$ and $\widetilde{x}_0^*=(2/3,1/3)$ are always critical points.

The optimality of $\Lambda_2$ among lattices already follows from the seminal work of Montgomery~\cite{Mon88}. Corollary \ref{cor:2d} provides more evidence for the widely conjectured universal optimality of the hexagonal lattice $\Lambda_2$. Recently, it was noted by Coulangeon and Sch\"urmann \cite{CouSch22} that their proof of the hexagonal lattice being a local minimizer within the set of periodic configurations \cite{CouSch12} unfortunately contains a non-repairable flaw. There is practically no doubt on the correctness of their statement, after all the hexagonal lattice is conjectured to be the unique \textit{global} minimizer among periodic configurations. However, to the best of our knowledge, the status of universal optimality of the hexagonal lattice has been set back to the result of Montgomery \cite{Mon88} from 1988. In this note, we now extend the currently known result by enlarging the set of configurations within which the hexagonal lattice is optimal.

Universal optimality is a rare property and so far it was only proved to exist in dimensions 1, 8, and 24. Moreover, it provably fails for lattices in dimensions 3, 5, 6, and 7 (cf.\ \cite[Sec.~9]{CohKum07}). The reason is quite simple to state: to be universally optimal a lattice needs to solve the sphere packing problem, which can be derived as a limiting case of energy minimization. Among lattices, the optimal sphere packing solutions in the named dimensions are known \cite[Table~1.1]{ConSlo98}. By the Poisson summation formula, one can then prove that if a lattice is optimal for a parameter $\alpha$ (density fixed to 1), then its dual lattice is optimal for the parameter $1/\alpha$. As none of the optimal sphere packing lattices in the named dimensions are self-dual (up to scaling), they cannot be universally optimal. For a general potential, the circumstance that at different densities different configurations are energy-minimizing is called \textit{phase transition}. Now, universal optimality includes the statement that phase transition cannot happen for completely monotone potentials.

The integer lattice $\Lambda_1$, the hexagonal lattice $\Lambda_2$ ($\mathsf{A}_2$-root lattice), the $\mathsf{D}_4$ root lattice $\Lambda_4$, the $\mathsf{E}_8$ root lattice $\Lambda_8$ and the Leech lattice $\Lambda_{24}$ are self-dual (up to scaling). Here, we used the $\Lambda_d$ notation of \cite[Table 1.2]{ConSlo98}. It was proven by Cohn and Kumar that the (scaled) integer lattice is universally optimal in dimension 1 \cite{CohKum07}. Recently, Cohn, Kumar, Miller, Radchenko, and Viazovska \cite{Coh-Via22} proved universal optimality of $\Lambda_8$ and $\Lambda_{24}$. In dimension 4, numerical evidence was provided by Cohn, Kumar, and Sch\"urmann \cite{CohKumSch09} that $\Lambda_4$ may be universally optimal as well.

\section{Notation}
We are interested in the case when the potential is a Gaussian of width~$\alpha$, i.e.,
\begin{equation}
    \phi_\alpha (r) = e^{-\alpha r^2}, \quad r \in \R_+, \, \alpha>0.
\end{equation}
The Fourier transform of a (suitable) function $f$ is given by
\begin{equation}
    \F f(\omega) = \int_{\R^d} f(t) e^{-2 \pi i \omega \cdot t} \, dt.
\end{equation}
The dot $\cdot$ denotes the Euclidean inner product of two vectors (always seen as column vectors). The Poisson summation formula for a lattice $\Lambda$ with dual lattice $\Lambda^\perp$ reads 
\begin{equation}
    \sum_{\lambda \in \Lambda} f(\lambda + x) = \frac{1}{\vol(\R^d/\Lambda)} \sum_{\lambda^\perp \in \Lambda^\perp} \F f(\lambda^\perp) e^{2 \pi i \lambda^\perp \cdot x}.
\end{equation}
Recall that the behavior of the Gaussian under the Fourier transform is
\begin{equation}
    \F \phi_{\pi \alpha} = \alpha^{-1/2} \phi_{\pi/\alpha}.
\end{equation}
Since the additional factor $\pi$ is most often not relevant for our results we generally omit it. When it becomes important in a specific calculation we will explicitly write it. Moreover, fixing the density to 1 is not a restriction and we may choose any other fixed density. For brevity, we demonstrate this for a lattice $\Lambda = \rho^{-d/2} \Lambda_0 \subset \R^d$ of density $\rho$, where $\Lambda_0$ has density 1. For all $\alpha > 0$ we have
\begin{equation}\label{eq:density}
    E_{\phi_\alpha}(\Lambda) = E_{\phi_\alpha}(\rho^{-d/2} \Lambda_0) = \sum_{\lambda \in \rho^{-d/2} \Lambda_0} e^{-\alpha |\lambda|^2} = \sum_{\lambda \in \Lambda_0} e^{-\frac{\alpha}{\rho^d} |\lambda|^2} = E_{\phi_{\alpha/\rho^d}}(\Lambda_0).
\end{equation}
As the results shall hold for all $\alpha > 0$, we can choose any density we like and hide it in the exponent of the Gaussian. So, if a result holds for one fixed density and all $\alpha$, then it already holds for all densities (and all $\alpha$). In the sequel, we set $E(\Gamma) = E_{\phi_\alpha}(\Gamma)$ to simplify notation.
Moreover, we will identify configurations which only differ by rotation or reflection (as they have the same energy). We write
\begin{equation}
    \Gamma \cong \widetilde{\Gamma}
    \quad \text{ if there exists an orthogonal matrix $\mathcal{Q}$ such that } \quad
    \Gamma = \mathcal{Q} \, \widetilde{\Gamma}.
\end{equation}
So, most often, explicit co-ordinates are not relevant but for some specific situations picking a certain lattice basis simplifies calculations. To this end, we note that we prefer the hexagonal lattice to be written as (recall that we use column vectors)
\begin{equation}
    \Lambda_2 = \sqrt{\tfrac{2}{\sqrt{3}}}
    \begin{pmatrix}
        1 & \frac{1}{2}\\[.6ex]
        0 & \frac{\sqrt{3}}{2}
    \end{pmatrix} \Z^2.
\end{equation}

\section{Auxiliary results}
\subsection{Jacobi theta function}\label{sec:theta}
We will need a result on Jacobi theta functions in the sequel. For argument $z \in \C$ and parameter $\tau \in \mathbb{H} = \{z \in \C \mid \Im(z) > 0 \}$, denote the classical Jacobi theta-3 function by
\begin{equation}
    \vartheta_3(z;\tau) = \sum_{k \in \Z} e^{\pi i \tau k^2} e^{2 \pi i k z}.
\end{equation}
It has a product representation known as the Jacobi triple product \cite{SteSha03_Complex}, \cite{WhiWat69};
\begin{align}
    \vartheta_3(z;\tau)
    & = \prod_{k \geq 1} \left(1 - e^{2 k \pi i \tau} \right) \left(1 + e^{(2k-1)\pi i \tau} e^{2 \pi i z} \right) \left(1 + e^{(2k-1)\pi i \tau} e^{-2 \pi i z} \right)\\
    & = \prod_{k \geq 1} \left(1 - e^{2 k \pi i \tau} \right) \left(1 + 2 e^{(2k-1)\pi i \tau} \cos(2 \pi z) + e^{2(2k-1)\pi i \tau} \right).
\end{align}
As we will only need the real-valued version of $\vartheta_3$, we set
\begin{equation}\label{eq:theta_3}
    \theta_3(x,\alpha) = \vartheta_3(x,i \alpha), \quad x \in \R, \, \alpha > 0.
\end{equation}
\begin{lemma}\label{lem:theta_triple}
    Let $\alpha > 0$ be fixed. Then,
    \begin{equation}
        \theta_3(0;\alpha) \geq \theta_3(x;\alpha) \geq \theta_3(\tfrac{1}{2};\alpha),
    \end{equation}
    with equality if and only if $x \in \Z$ for the first and $x \in \Z + \frac{1}{2}$ for the second inequality.
\end{lemma}
\begin{proof}
    The proof follows easily from the product representation and the fact that $\cos(2 \pi x)$ assumes its maximum for $x \in \Z$ and its minimum for $x \in \Z + \frac{1}{2}$.
\end{proof}
This fact has already been noted and used by Janssen in his work on Gabor frame bounds~\cite{Jan96}. A finer analysis of the behavior of $\theta_3(x;\alpha)$ has been carried out by Montgomery \cite{Mon88}.

\subsection{The result of Cohn and Kumar}

In dimension 1, the considered energy minimization problem was completely settled by Cohn and Kumar \cite{CohKum07}. Their result shows that at any density the scaled integers are universally optimal in dimension~1.
\begin{theorem}[Cohn and Kumar, 2007]\label{thm:Cohn-Kumar}
    Among all periodic configurations of the form $\Gamma = \bigcup_{k=1}^N \delta\Z + x_k$ of fixed density, with $x_k \in [0;\delta)$, $x_k \neq x_j$ for $k \neq j$
    \begin{equation}
        E(\Gamma) = \frac{1}{N} \, \sum_{j,k=1}^N \sum_{\ell \in \Z} \phi_\alpha(|\ell+x_k-x_j|)
    \end{equation}
     is minimal if and only if $x_k = \frac{k-1}{N}+x$ for a global shift $x \in \R$, i.e., $\Gamma$ is equispaced.
\end{theorem}
The term \textit{universally optimal} refers to the fact that the result for the Gaussian function $\phi_\alpha$ implies that the same holds for any completely monotone function of squared distance. A function $f: \R_+ \to \R_+$ is completely monotone if and only if
\begin{equation}
    (-1)^{k} f^{(k)}(r) \geq 0 \quad \text{ for all } k \in \N.
\end{equation}
By the Bernstein-Widder theorem (\cite{Ber28}, \cite{Wid41}) any completely monotone function can also be written as the Laplace transform of a non-negative Borel measure $\mu$.
\begin{theorem}[Bernstein-Widder, 1928/1941]\label{thm:Bernstein-Widder}
    Let $f$ be completely monotone. Then there exists a non-negative Borel measure $\mu_f$ such that
    \begin{equation}
        f(r) = \int_0^\infty e^{-\alpha r} \, d\mu_f(\alpha).
    \end{equation}
\end{theorem}
It follows that, for any dimension $d$, if a configuration is optimal for the Gaussian $\phi_\alpha$ for all $\alpha > 0$, then it is also optimal for all completely monotone potential functions $p$ of squared distance:
\begin{align*}
    \frac{1}{N} \sum_{j,k=1}^N  \sum_{\lambda\in\Lambda\setminus\{0\} } p(|\lambda+x_j-x_k|^2)
    & = \frac{1}{N} \sum_{j,k=1}^N \sum_{\lambda\in\Lambda\setminus\{0\} } \int_0^\infty e^{-\alpha |\lambda+x_j-x_k|^2} \, d\mu_p(\alpha)\\
    & = \frac{1}{N} \int_0^\infty \bigg( \sum_{j,k=1}^N \sum_{\lambda\in\Lambda\setminus\{0\} } e^{-\alpha |\lambda+x_j-x_k|^2} \bigg) d\mu_p(\alpha).
\end{align*}
Note that, by Bernstein \cite[Thm.~9.16]{Ber28}, the above integral is indeed convergent and the interchange of summation and integration is justified by monotone convergence: the potential is an increasing limit of weighted sums of Gaussians (cf.~\cite[Sec.~1]{Coh-Via22}).

\subsection{The result of Montgomery}
In dimension 2, the energy minimization problem is solved among \textit{lattices} by the result of Montgomery \cite{Mon88}. To the best of our knowledge, this is the best available result in dimension 2 after the appearance of \cite{CouSch22}. Before stating the result, we introduce the following (theta-like) function:
\begin{equation}\label{eq:f_Gamma}
    f_\Gamma(x;\alpha) = \sum_{\gamma \in \Gamma} \phi_\alpha(|x-\gamma|) = \sum_{\gamma \in \Gamma} e^{-\alpha |x-\gamma|^2} , \quad \alpha > 0.
\end{equation}
We will often write $f_\Gamma(x)$ instead of $f_\Gamma(x;\alpha)$ if we do not need a specific $\alpha$.
\begin{theorem}[Montgomery, 1988]\label{thm:Montgomery}
    For all $\alpha > 0$ and among all lattices of any fixed density, the following inequality (for lattice theta functions) holds
    \begin{equation}
        f_\Lambda(0) \geq f_{\Lambda_2}(0), \quad \forall \alpha > 0,
    \end{equation}
    with equality if and only if $\Lambda$ is the hexagonal lattice $\Lambda_2$.
\end{theorem}
As the energy of a lattice $\Lambda$ can be written as $E_{\phi_\alpha}(\Lambda) = f_\Lambda(0)-1$, we see that Montgomery's result implies energy minimization for the Gaussian. By Theorem \ref{thm:Bernstein-Widder}, the same holds for completely monotone potentials of squared distance.

\section{Proofs}
\subsection{Proof of Theorem \ref{thm:tensor}}
\begin{proof}[Proof of Theorem \ref{thm:tensor}]
    If $\Gamma_\times$ is as in \eqref{eq:Cartesian_thm}, i.e., it is the Cartesian product of 1-dimensional periodic configurations $\Gamma_l$, $l = 1, \ldots d$, then it can be described as
    \begin{equation}\label{eq:Cartesian_periodic} 
		\Gamma_\times = \bigcup\limits_{l=1}^d \bigcup\limits_{j_l=1}^{J_l} \delta \Z^{d} +(x_{1,j_1}, x_{2,j_2}, \, \ldots \, , x_{d, j_d}),
    \end{equation}
    where $\delta = \diag(\delta_1, \, \ldots \, , \delta_d)$ denotes the diagonal matrix with entries $\delta_1, \, \ldots \, ,\delta_d$. Using the factorization properties of the exponential function, the energy \eqref{eq:energy} for $\Gamma_\times$ can be written as follows:
    \begin{align*}
        E(\Gamma_\times) & = \frac{1}{\prod\limits_{m=1}^d J_m} \, \sum\limits_{l=1}^d \sum\limits_{j_l,k_l = 1}^{J_l}
        \sum\limits_{\lambda\in \Z^d\setminus\{0\} } \phi_\alpha(|\delta\lambda+(x_{1,j_1} - x_{1,k_1}, \, \ldots \, , x_{d,j_d}-x_{d,k_d})|) \\
        & = \frac{1}{\prod\limits_{m=1}^d J_m} \, \sum\limits_{l=1}^d \sum\limits_{j_l,k_l = 1}^{J_l}
        \sum\limits_{\lambda\in \Z^d\setminus\{0\}} \prod\limits_{m=1}^d \phi_\alpha(|\delta_m \lambda_m +x_{m,j_m} - x_{m,k_m}|) \\
        & = \prod\limits_{m=1}^d \bigg(\frac{1}{J_m}
        \sum\limits_{j_m,k_m = 1}^{J_m} \sum\limits_{\lambda\in \Z\setminus\{0\}} \phi_\alpha(|\delta_m \lambda +x_{m,j_m} - x_{m,k_m}|) \bigg).
    \end{align*}
    This shows that the problem reduces to finding the optimal point configuration in each coordinate, i.e., we can refer to the one-dimensional solution.
    By the result of Cohn and Kumar \cite{CohKum07} (Theorem \ref{thm:Cohn-Kumar}) it follows that the optimal solutions need to be equispaced in each of the $\Gamma_l$, i.e., it has to be a cuboid-shaped lattice of the form
    \begin{equation}
        \Lambda_{\beta} = \beta_1 \Z \times \ldots \times \beta_d \Z,
    \end{equation}
    with $\beta_m = \delta_m/J_m$, $m = 1, \ldots, d$. For the lattice $\Lambda_\beta$, the energy for the Gaussian potential can be written as a product of Jacobi theta functions:
    \begin{equation}
        E(\Lambda_\beta) = \prod_{j=1}^d \sum_{k \in \Z} e^{-\alpha \beta_j^2 k^2} - 1 = \prod_{j=1}^d \theta_3\left(0; \tfrac{\alpha \beta_j^2}{\pi} \right) - 1, \quad \alpha > 0.
    \end{equation}
    It was also proven by Montgomery \cite{Mon88} (see also \cite{BetPet17}, \cite{FauSte17}) that the above product is minimal if and only if $\beta_1 = \ldots = \beta_d$, when $\prod_{k=1}^d \beta_k$ is fixed.
\end{proof}

\subsection{The union of 2 lattices}
Lattices are the simplest periodic structures. The next natural step is to consider the union of two lattices in $\R^d$;
\begin{equation}
    \Gamma = \Lambda \cup \left( \Lambda + x \right), \quad x \notin \Lambda.
\end{equation}
This already turns out to be tricky very quickly as, in general, it is awfully difficult to determine the shift (which depends on $(d^2+d-2)/2$ lattice parameters and $\alpha$) of the second lattice which minimizes the energy of the configuration.

\subsubsection{The union of two cuboid-shaped lattices}
The simplest case for the union of 2 lattices is the following: let $\Lambda = \delta \Z^d$, where $\delta = \diag(\delta_1, \ldots, \delta_d)$. Consider the periodic configuration of the following form:
\begin{equation}\label{eq:Gamma_x}
    \Gamma_x = \Lambda \cup (\Lambda + \delta x), \quad x \notin \Z^d.
\end{equation}
Then, as we will show, the energy of the system can be written as
\begin{equation}\label{eq:E_Gamma_x}
    E(\Gamma_x) = E(\Lambda) + f_\Lambda(x),
\end{equation}
where $f_\Lambda(x)$ is defined as in \eqref{eq:f_Gamma}. Equation \eqref{eq:E_Gamma_x} actually holds for any lattice, not only cuboid-shaped ones. For a lattice of type \eqref{eq:Gamma_x}, it follows from the behavior of $f_\Lambda$ and Lemma \ref{lem:theta_triple} (cf.\ \cite[Lemma~1]{Mon88}) that
\begin{equation}
    E(\Gamma_x) \geq E(\Lambda) + f_\Lambda(\delta \mathbf{\tfrac{1}{2}}), \quad \mathbf{\tfrac{1}{2}} = (\tfrac{1}{2}, \, \ldots \, , \tfrac{1}{2})
\end{equation}
with equality being attained if and only if $x \in \mathbf{\tfrac{1}{2}} + \Z^d$, as we will now show.
\begin{proposition}\label{pro:energy_2_lattices}
    Let $\Gamma_x$ be defined as in \eqref{eq:Gamma_x}. The energy of $\Gamma_x$ is given by \eqref{eq:E_Gamma_x} and is minimal if and only if $x \in \mathbf{\frac{1}{2}} + \Z^d$. Hence, the minimizer is a lattice.
\end{proposition}
\begin{proof}
    First, allow $\Lambda$ to be any lattice. Set $\Gamma_x = \Lambda \cup (\Lambda + x)$, $x_1 = 0$ and $x_2 = x \notin \Lambda$. Note that $f_\Lambda(x) = f_\Lambda(-x)$ for any lattice $\Lambda$. Then, we can write the energy as
    \begin{align}
        E(\Gamma_x)
        & = \frac{1}{2} \sum_{\gamma, \gamma' \in \Gamma_x} e^{-\alpha |\gamma-\gamma'|^2} - 1
        = \frac{1}{2} \sum_{j,k=1}^2 \sum_{\lambda \in \Lambda} e^{-\alpha |\lambda + x_j - x_k|^2} - 1\\
        & = \sum_{\lambda \in \Lambda} e^{-\alpha |\lambda|^2} -1 + \sum_{\lambda \in \Lambda} e^{-\alpha |\lambda + x|^2}
        = E(\Lambda) + \sum_{\lambda \in \Lambda} e^{-\alpha|\lambda+x|^2}\\
        & = E(\Lambda) + f_\Lambda(x).
    \end{align}
    Due to the special structure of $\Lambda = \delta \Z^d$,
    we can write the energy as
    \begin{equation}
        E(\Gamma_x)
        = E(\Lambda) + \prod_{k=1}^d \sum_{n \in \Z} e^{- \alpha \delta_k^2 (n + x_k)^2}
        = E(\Lambda) + \prod_{k=1}^d \sqrt{\tfrac{\pi}{{\alpha \delta_k^2}}} \, \theta_3 \left( x_k;\tfrac{\pi}{\alpha \delta_k^2} \right).
    \end{equation}
    For the second equality, we used the Poisson summation formula. Now, we can use Lemma \ref{lem:theta_triple} to see that $E(\Gamma_x)$ is minimal if and only if $x \in \mathbf{\frac{1}{2}} + \Z^d$.
\end{proof}

\subsubsection{The union of two hexagonal lattices.}
As the hexagonal lattice is outstanding in the set of 2-dimensional lattices, we consider periodic configurations of the form
\begin{equation}\label{eq_union hexagons}
    \Gamma_2(x) = \Lambda_2 \cup (\Lambda_2+x).
\end{equation}
Shifting $\Lambda_2$ by a deep hole $x_0$ gives the so-called honeycomb structure (see Figure~\ref{fig:honeycomb})
\begin{equation}
    \Gamma_{\textnormal{\hexagon}} = \Lambda_2 \cup \left(\Lambda_2 + x_0\right).
\end{equation}
\begin{figure}[h!t]
    \includegraphics[width=.45\textwidth]{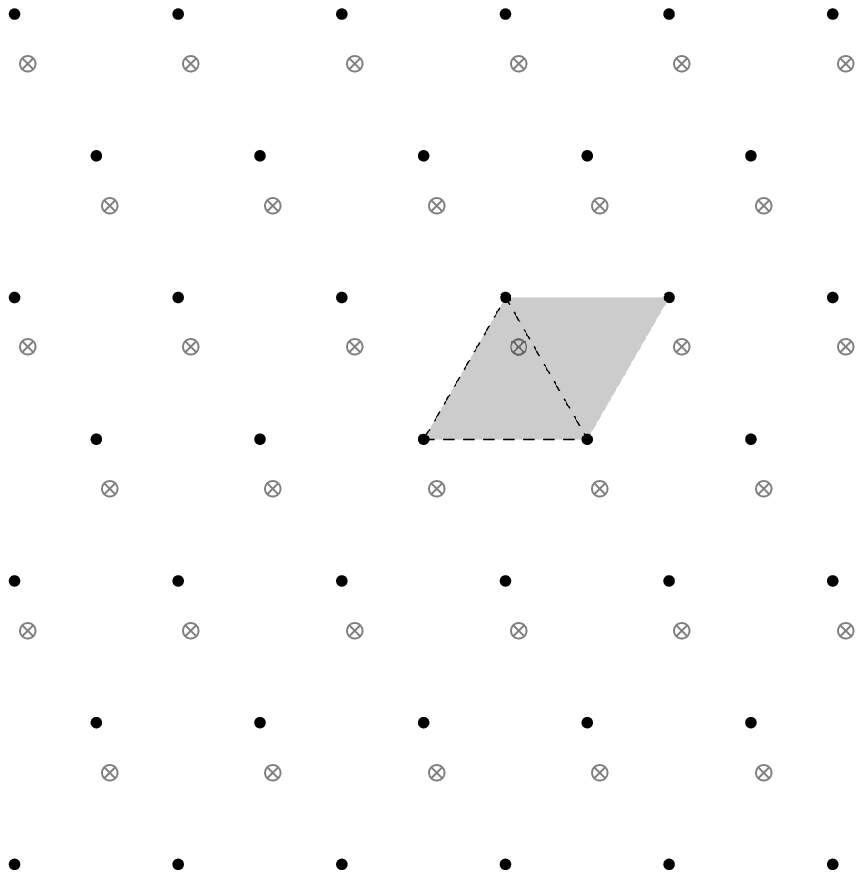}
    \hfill
    \includegraphics[width=.45\textwidth]{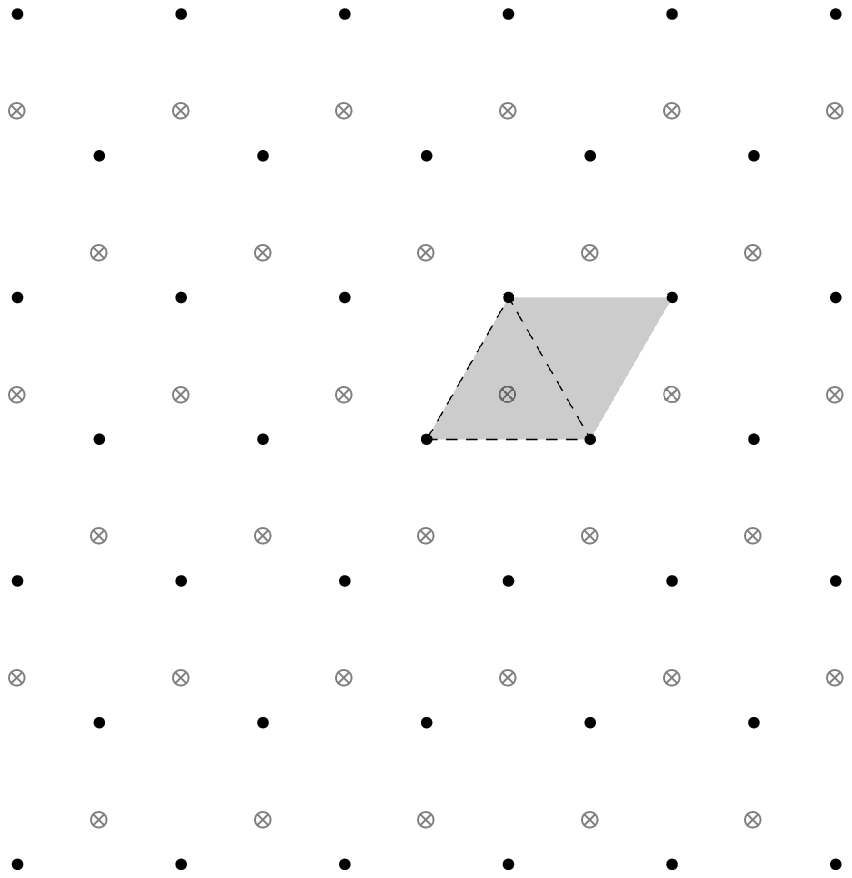}
    \caption{Union of two hexagonal lattices by an arbitrary shift (left) and a shift by a deep hole (right). A deep hole is the center of a fundamental triangle and we have two deep holes in a fundamental parallelogram. The resulting configuration is known as the honeycomb structure, which is not a lattice.}
    \label{fig:honeycomb}
\end{figure}
Coordinates of a \textit{deep hole} (cf.~\cite[Chap.~1]{ConSlo98}) in the hexagonal lattice are given by
\begin{equation}\label{eq:deep_hole}
    x_0 = \sqrt{\tfrac{2}{\sqrt{3}}}
    \begin{pmatrix}
        1 & \frac{1}{2}\\
        0 & \frac{\sqrt{3}}{2}
    \end{pmatrix}
    \begin{pmatrix}
        k + \frac{m}{3}\\
        l + \frac{m}{3}
    \end{pmatrix},
    \quad k,l \in \Z, \, m \in \{1,2\}.
\end{equation}
\begin{proposition}\label{pro:honeycomb}
    Among all configurations $\Gamma_2(x)$ of the form \eqref{eq_union hexagons}, the honeycomb structure $\Gamma_\textnormal{\hexagon}$ is the unique minimizer of $E(\Gamma_2(x))$.
\end{proposition}
\begin{proof}
    The proof requires a result of Baernstein \cite{Bae97} on the minimal temperature of the heat kernel on the hexagonal flat torus $\R^2/\Lambda_2$. It states that
    \begin{equation}
        f_{\Lambda_2}(x) = \sum_{\lambda \in \Lambda_2} e^{-\alpha |\lambda+x|^2} \geq \sum_{\lambda \in \Lambda_2} e^{-\alpha |\lambda+x_0|^2} = f_{\Lambda_2}(x_0),
    \end{equation}
    with equality if and only if $x$ is another deep hole of $\Lambda_2$. The result now follows as
    \begin{equation}
        E(\Gamma_2(x)) = E(\Lambda_2) + f_{\Lambda_2}(x) \geq E(\Lambda_2) + f_{\Lambda_2}(x_0) = E(\Gamma_{\textnormal{\hexagon}}).
    \end{equation}
\end{proof}

We remark that the hexagonal lattice cannot be written as a union of two shifted hexagonal lattices. Regarding the conjectured universal optimality of the hexagonal lattice, the honeycomb $\Gamma_{\textnormal{\hexagon}}$ is probably the most outstanding rival in the set of periodic configurations which is not a lattice. We will now prove that, scaled to unit density, the honeycomb structure cannot beat the hexagonal lattice $\Lambda_2$.

\subsection{Proof of Theorem \ref{thm:hex_vs_honeycomb}}

We start by proving properties of two auxiliary functions. These will be used for refined convexity estimates of the energy.

\begin{wrapfigure}[0]{r}[10pt]{0.6\textwidth}
	\vskip-55pt
    \centering
    \includegraphics[width=.5\textwidth]{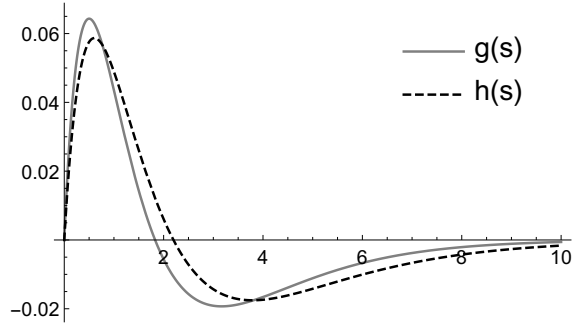}
    \caption{The auxiliary functions $g(s)$ and $h(s)$.}
    \label{fig:f(s)_g(s)}
\end{wrapfigure}

\vskip45pt

\begin{align}
    g(s) & = e^{-s}-\tfrac{1}{2}e^{-2s}-\tfrac{1}{2} e^{-2/3s},\label{eq:g}\hskip170pt \\
    h(s) & = e^{-s} - \tfrac{1}{4}e^{-s/2} - \tfrac{3}{4}e^{-3/2s}. \hskip170pt\label{eq:h}
\end{align}

\vskip65pt

\begin{lemma}\label{lemma:aux_estimate} The following estimates hold:
    \begin{equation}
    g(s) < 0 \quad \text{ for } \; 2<s<\infty
    \quad \text{ and } \quad
    h(s) < 0 \quad \text{ for } \; 6<s<\infty.
    \end{equation}
\end{lemma}
\begin{proof}
    We start with the inequality for $g$. Setting $t=e^{-s/3}$, it suffices to verify
    $$
     -t^6+2t^3-t^2 < 0 \quad \text{ for  } t \in (0,e^{-2/3}).
    $$
    Note that for $0<t\le e^{-2/3}<1$ we have
    $$-1 \, < \, t^3+t^2+t-1 \, < \, e^{-2}+e^{-4/3}+e^{-2/3}-1 \, < \, -0.087.$$
    Therefore, since
    $$
    -t^6+2t^3-t^2 = -t^2(t-1)(t^3+t^2+t-1),
    $$
    we get the required inequality for $g$. For the function $h$, we set $t= e^{-s/2}$ and verify
    \begin{equation}
        -3t^3+4t^2-t<0 \quad \text{ for } t \in (0,e^{-3}).
    \end{equation}
    Since  $(0,e^{-3}) \subset (0, {1/3})$ we obtain
    \begin{equation}
       -3t^3+4t^2-t = -t(1-t)(1-3t) <0.
    \end{equation}
\end{proof}

\begin{proof}[Proof of Theorem \ref{thm:hex_vs_honeycomb}]
    This time we write the energies with $\pi$ in the exponent;
    \begin{align}
        E_{\phi_{\pi \alpha}}(\Lambda_2)
        & = \sum_{k,l \in \Z} e^{-\frac{2}{\sqrt{3}} \pi \alpha (k^2+kl+l^2)} - 1,\\
        E_{\phi_{\pi \alpha}}(\Gamma_\textnormal{\hexagon})
        & = E_{\phi_{\pi \alpha}}(\sqrt{2}\Lambda_2) + \sum_{\lambda \in \Lambda_2} e^{-\pi (2\alpha) |\lambda+x_0|^2}\\
        & = E_{\phi_{\pi \alpha}}(\sqrt{2}\Lambda_2) + \frac{1}{2 \alpha} \sum_{\lambda \in \Lambda_2} e^{-\frac{\pi}{2\alpha} |\lambda|^2} e^{2 \pi i \sigma(\lambda,x_0)}.
    \end{align}
    For the last equality, we used the symplectic version of the Poisson summation formula. This is a nice bookkeeping trick in dimension 2 (see \cite[Sec.~1.1]{BetFau23}, \cite[App.~B]{BetFauSte21}, \cite[Sec.~4.4]{FauGumSha22}) as the dual lattice is replaced by a 90 degrees rotated version called the adjoint lattice $\Lambda^\circ$. In particular, in dimension 2 any lattice (of density 1) is its own adjoint (up to scaling for density other than 1). This also requires a change of variables in the complex exponential. The Euclidean inner product $\lambda \cdot x$ is replaced by the skew-symmetric form $\sigma(\lambda, x)$;
    \begin{equation}
        \sigma(\lambda,x) = \lambda \cdot J x,
        \quad \text{ where } \quad
        \lambda, \, x \in \R^2, \; J =
        \begin{pmatrix*}[r]
            0 & -1\\
            1 & 0
        \end{pmatrix*},
        \quad
        \Lambda^\circ = J \Lambda^\perp.
    \end{equation}
    Now we can use the cubic AGM of Borwein and Borwein \cite{BorBor91} to write the (symplectic) Fourier series in terms of the energy on scaled hexagonal lattices. For the sake of self-containment, we present necessary parts from \cite{BorBor91}. For $|q|<1$, $\zeta_3^3=1$, $\zeta_3 \neq 1$~set
    \begin{equation}
        a(q) = \sum_{k,l \in \Z}q^{k^2+kl+l^2}, \quad b(q) = \sum_{k,l\in\Z} q^{k^2+kl+l^2} \zeta_3^{k-l},
        \quad \text{ then } \quad
        3 a(q^3) = a(q) + 2 b(q).
    \end{equation}
    At this point, we also refer to \cite{FauGumSha22} for some geometric explanations. By using the above equation, which is the arithmetic part of the cubic AGM, we get
    \begin{align}
         E_{\phi_{\pi \alpha}}(\Gamma_\textnormal{\hexagon})
         & = E_{\phi_{\pi \alpha}}(\sqrt{2} \Lambda_2) + \frac{1}{4 \alpha} \left( 3 (E_{\phi_{3\pi/(2\alpha)}}(\Lambda_2)+1) - (E_{\phi_{\pi/(2\alpha)}}(\Lambda_2) +1)\right)\\
         & = E_{\phi_{\pi \alpha}}(\sqrt{2} \Lambda_2) + \frac{1}{2} \left( (E_{\phi_{(2\alpha)\pi/3}}(\Lambda_2)+1) - (E_{\phi_{(2\alpha) \pi}}(\Lambda_2) +1)\right)\\
         & = E_{\phi_{\pi \alpha}}(\sqrt{2} \, \Lambda_2) + \frac{1}{2} \left( E_{\phi_{\alpha\pi}}(\sqrt{2/3} \, \Lambda_2) - E_{\phi_{\alpha \pi}}(\sqrt{2}\Lambda_2) \right)\\
         & = \frac{1}{2} \left( E_{\phi_{\pi \alpha}}(\sqrt{2} \, \Lambda_2) + E_{\phi_{\alpha\pi}}(\sqrt{2/3} \, \Lambda_2)\right),
    \end{align}
    where we used the Poisson summation formula in the second step and the ``density trick" from \eqref{eq:density} in the third step. We need to show that, for all $\alpha > 0$, we have
    \begin{equation}\label{eq:goal1}
        E_{\phi_{\alpha\pi}}(\Lambda_2) < \frac{1}{2} \left( E_{\phi_{\pi \alpha}}(\sqrt{2} \, \Lambda_2) + E_{\phi_{\alpha\pi}}(\sqrt{2/3} \, \Lambda_2)\right),
    \end{equation}
    which by the Poisson summation formula is equivalent to 
    \begin{equation}\label{eq:goal2}
         E_{\phi_{\pi/\alpha}}(\Lambda_2) < \frac{1}{2} \left( \frac{1}{2}E_{\phi_{\pi /\alpha}}(\sqrt{1/2} \,\Lambda_2) + \frac{3}{2} E_{\phi_{\pi/\alpha}}(\sqrt{3/2} \, \Lambda_2)\right).
    \end{equation}
    So, the hexagonal lattice $\Lambda_2$ of density 1 has smaller energy than a weighted average of its scaled versions of density $1/2$ and $3/2$. This is almost a consequence of the convexity of the energy in $\alpha$, but not quite. Writing out the formulas, one sees that we need a kind of convexity result in $1/\alpha$. We proceed by comparing terms where the quadratic form $k^2+kl+l^2$ is constant in inequalities \eqref{eq:goal1} and \eqref{eq:goal2}. Observe that the inequality in  \eqref{eq:goal1} can be rewritten as
    \begin{align}\label{eq:goal1_1}
        \sum\limits_{(k,l)\in\Z^2 \setminus\{ (0,0)\}} g\left(\tfrac{2\pi \alpha}{\sqrt{3}} (k^2+kl+l^2)\right) < 0,
    \end{align}
    where $f$ is defined as in \eqref{eq:g} and the inequality in \eqref{eq:goal2} is equivalent to 
    \begin{align}\label{eq:goal2_2}
        \sum\limits_{(k,l)\in\Z^2 \setminus\{ (0,0)\}} h\left(\tfrac{2\pi}{\sqrt{3}\alpha} (k^2+kl+l^2)\right) < 0,
    \end{align}
    where $g$ is defined as in \eqref{eq:h}. It clearly suffices to prove that each summand is negative to conclude the claim. We make two case distinctions.
    
    \textbf{Case 1: $\alpha\geq\frac{\sqrt{3}}{\pi}$.} In this case, for all $(k,l)\in\Z^2\setminus \{(0,0)\}$ we have
    \begin{equation}
        \tfrac{2\pi \alpha}{\sqrt{3}} (k^2+kl+l^2) \geq 
        \tfrac{2\pi}{\sqrt{3}}\cdot\tfrac{\sqrt{3}}{\pi} =2,    
    \end{equation}
     so by Lemma \ref{lemma:aux_estimate}, the inequality \eqref{eq:goal1_1} holds for all $\alpha\geq \frac{\sqrt{3}}{\pi}$.
    
     \textbf{Case 2: $\alpha\leq \frac{\sqrt{3}}{\pi}$.} In this case,
     \begin{equation}
        \tfrac{2\pi}{\sqrt{3}\alpha}(k^2+kl+l^2) \geq 
        \tfrac{2\pi}{\sqrt{3}}\cdot\tfrac{\pi}{\sqrt{3}} = \tfrac{2\pi^2}{3}>6,    
    \end{equation}
     so by Lemma \ref{lemma:aux_estimate}, the inequality \eqref{eq:goal2_2} holds for all $\alpha\leq \frac{\sqrt{3}}{\pi}$. In combination, this means that \eqref{eq:goal1} and \eqref{eq:goal2} hold for all $\alpha > 0$.    
\end{proof}

\begin{proof}[Proof of Corollary \ref{cor:2d}]
    The proof is just accumulating existing and derived results.

    \smallskip
    \noindent
    - The set $\mathbf{\Lambda}$: by Theorem \ref{thm:Montgomery}, $\Lambda_2$ is optimal in $\mathbf{\Lambda}$ for all $\alpha >0$.

    \smallskip
    \noindent
    - The set $\mathbf{\Gamma}_\times$: by Theorem \ref{thm:tensor}, which is a consequence of Theorem \ref{thm:Cohn-Kumar}, we know that the optimal structure in the set $\mathbf{\Gamma}_\times$ needs to be a rectangular lattice. Among these, the square lattice yields the smallest energy (\cite{Mon88}, see also \cite{FauSte17}) and by Theorem~\ref{thm:Montgomery} the hexagonal lattice has smaller energy.

    \smallskip
    \noindent
    - The set $\mathbf{\Gamma}_2$: let $\Gamma_x$ be the union of two rectangular lattices as in \eqref{eq:Gamma_x}. For $x_0 = \delta \mathbf{\frac{1}{2}}$, it is not hard to see that $\Gamma_{x_0}$ is a lattice
    \begin{equation}
        \Gamma_{x_0} = \delta \Z^2 \cup \delta \left(\Z^2 + \mathbf{\frac{1}{2}}\right) =
        \begin{pmatrix}
            \delta_1 & \frac{\delta_1}{2}\\[.6ex]
            0 & \frac{\delta_2}{2}
        \end{pmatrix}
        \Z^2.
    \end{equation}
    By Theorem \ref{thm:Montgomery} the hexagonal lattice minimizes energy among lattices and by Proposition \ref{pro:energy_2_lattices}, the minimizing configuration $\Gamma_x$ needs to be the lattice $\Gamma_{x_0}$. The hexagonal lattice can be written as a union of two rectangular lattices with $\delta_2/\delta_1 = \sqrt{3}$. Now let $\Gamma_2(x)$ be the union of two hexagonal lattices. By Proposition \ref{pro:honeycomb}, the optimal structure is the honeycomb $\Gamma_\textnormal{\hexagon}$ and by Theorem \ref{thm:hex_vs_honeycomb}, it has higher energy than the hexagonal lattice $\Lambda_2$ for all $\alpha > 0$.
\end{proof}

\subsection{Union of 3 lattices}
Analyzing the union of 3 arbitrary lattices is a too challenging endeavor for this note. So we will focus on some particular cases. Note that the union of 3 hexagonal lattices, i.e., $\Lambda_2 \cup (\Lambda_2 + x) \cup (\Lambda_2 + y)$ does not produce too many interesting results: the optimal shifts are the deep holes $x=x_0$ and $y=-x_0$ defined by \eqref{eq:deep_hole}. This means that the optimal configuration of 3 shifted hexagonal lattices is again a hexagonal lattice (of thrice the density). Therefore, we will concentrate on the union of 3 square lattices. In this case, the problem becomes surprisingly challenging. The energy of the set
\begin{equation}\label{eq:union3_Z2}
\Gamma_{x,y} = \Z^2 \cup (\Z^2+x)\cup (\Z^2+y)
\end{equation}
is, for $x=(x_1,x_2)$, $y=(y_1,y_2)$ and $f(x) = f_\Z(x;\alpha)$ ($f_\Z$ defined by \eqref{eq:f_Gamma}), given by 
\begin{equation}\label{eq:union3_xy_energy}
        E_{\phi_{\alpha \pi}}(\Gamma_{x,y}) = f(0)^2+\tfrac{2}{3} \left( f(x_1) f(x_2)+ f(y_1) f(y_2) + f(x_1-y_1) f(x_2-y_2) \right) -1.
\end{equation}
Thus, it suffices to determine the minimizers $(x,y)\in[0,1]^2$ of the objective function
\begin{equation}
    F(x,y)= f(x_1) f(x_2)+ f(y_1) f(y_2) + f(x_1-y_1) f(x_2-y_2).
\end{equation}
The critical points of $F$ are the solutions to the system of equations
\begin{align}
    f'(x_1) f(x_2) & = - f'(x_1-y_1) f(x_2-y_2), & &
    f(x_1) f'(x_2) = - f(x_1-y_1) f'(x_2-y_2), \\
    f'(y_1) f(y_2) & = f'(x_1-y_1) f(x_2-y_2), & &
    f(y_1) f'(y_2) = f(x_1-y_1) f'(x_2-y_2).
\end{align}
Adding the two left as well as the two right equations, we see that
\begin{equation}\label{eq:trivial_solutions}
    f'(y_1) f(y_2) = - f'(x_1) f(x_2)
    \quad \text{ and } \quad
    f(y_1) f'(y_2) = - f(x_1) f'(x_2).
\end{equation}
Since $f(x) = f(-x)$, we have $f'(x) = -f'(-x)$ and a trivial solution of the equations in \eqref{eq:trivial_solutions} is clearly given by $x=-y$. This prompted us to restrict to configurations of type $\Gamma_{x,-x}$. Note that the solutions of \eqref{eq:trivial_solutions} are periodic with respect to $\Z^2$. So we may assume $x \in [0,1]^2$ and to stay in the unit square, instead of $-x$, we use the solution $x^*=(1,1)-x$. The remaining restrictions are then given by
\begin{equation}\label{eq:critical_3points}
    f'(x_1)f(x_2) + f'(2 x_1)f(2 x_2) = 0
    \quad \text{ and } \quad
    f(x_1)f'(x_2) + f(2 x_1)f'(2 x_2) = 0.
\end{equation}
Due to symmetry and periodicity, solutions are given by the pairs
\begin{equation}\label{eq:solutions}
    x_0=(\tfrac{1}{3}, \tfrac{1}{3})
    \, \text{ and } \,
    x_0^*=(\tfrac{2}{3}, \tfrac{2}{3})
    \quad \text{ or by } \quad
    \widetilde{x}_0 = (\tfrac{1}{3}, \tfrac{2}{3})
    \, \text{ and } \,
    \widetilde{x}_0^* = (\tfrac{2}{3}, \tfrac{1}{3}).
\end{equation}
Thus, it is worthwhile to examine the energy of
\begin{equation}
    \Gamma_{x_0,x_0^*} = \Z^2 \cup (x_0+\Z^2)\cup (x_0^*+\Z^2) \cong \Gamma_{\widetilde{x}_0, \widetilde{x}_0^*}.
\end{equation}
\subsubsection{Open problem}
Is it true that (by periodicity $x,y \in [0,1]^2$) the configuration minimizing $E(\Gamma_{x,y})$ in \eqref{eq:union3_Z2} shows the symmetry $y = x^* = (1,1)-x$?
Assume that it indeed does.
Then the pairs $x_0$, $x_0^*$, and $\widetilde{x}_0$, $\widetilde{x}_0^*$ always yield critical configurations. What are other critical configurations? Take, e.g., $\alpha = 3.5$, $z=(1/4,1/2)$ and $z^*=(1,1)-z$ (note that $\Gamma_{z,z^*}$ is not critical). Using \textit{Mathematica} \cite{Mathematica} and its built-in function \textsf{EllipticTheta} (which can be evaluated to any precision), we get
\begin{equation}
    E_{\phi_{3.5\pi}}(\Gamma_{z,z^*}) \approx 0.17159 < 0.18279 \approx E_{\phi_{3.5\pi}}(\Gamma_{x_0,x_0^*}).
\end{equation}

\subsubsection{Asymptotic solution for small parameters}
We are now assuming that $y = x^*$. If $\alpha$ tends to 0, then we can prove that the only solutions in $[0,1]^2 \times [0,1]^2$ are indeed the pairs $(x_0,x_0^*)$ or $(\widetilde{x}_0,\widetilde{x}_0^*)$. We start with an auxiliary result for $f(x) = f_\Z(x;\alpha)$.
\begin{lemma}\label{lem:f_alpha}
    We define the following functions for $\alpha > 0$.
    \begin{equation}
        g(\alpha) =
        \begin{cases}
            \alpha^{-\frac{1}{2}} \left( 1 - \frac{22002}{11000} \, e^{-\frac{\pi}{\alpha}} \right) & \alpha < 1,\\
            2 e^{-\frac{\pi \alpha}{4}} & 1 \leq \alpha.
        \end{cases}
        \quad
        h(\alpha) =
        \begin{cases}
            \alpha^{-\frac{1}{2}} \left( 1 + \frac{22002}{11000} e^{-\frac{\pi}{\alpha}} \right) & \alpha < 1,\\
            1 + \frac{22002}{11000} e^{-\pi \alpha} & 1 \leq \alpha.
        \end{cases}
    \end{equation}
    Then, for any $x \in \R$ the following inequality is true:
    $
        g(\alpha) \leq f_\Z(x;\alpha) \leq h(\alpha).
    $
\end{lemma}
\begin{proof}
    As we have $f_\Z(1/2;\alpha) \leq f_\Z(x;\alpha) \leq f_\Z(0;\alpha)$ (Lemma \ref{lem:theta_triple}), it suffices to show the estimates for $x=1/2$ and $x=0$, respectively. We distinguish two cases:
    
    \noindent
    ($\alpha \geq 1$): For the upper estimate, we have
    \begin{align*}
        f_\Z(0;\alpha) 
        & = 1 + 2 \sum_{k \geq 1} e^{-\pi \alpha k^2} \leq 1 + 2 e^{-\pi \alpha} \sum_{k \geq 1} e^{-\pi (k^2-1)}
        \leq 1 + 2 e^{-\pi \alpha} \bigg( 1 + \sum_{k \geq 3} e^{-\pi k} \bigg)\\
        & = 1 + 2 e^{-\pi \alpha} \left( 1 + \tfrac{e^{-3 \pi}}{1 - e^{-\pi}} \right)
        \leq 1 + 2 e^{-\pi \alpha} \left( 1 + \tfrac{1}{11000} \right).
    \end{align*}
    As a lower estimate we simply use the leading term(s);
    \begin{equation}
        f_\Z(\tfrac{1}{2};\alpha) = 2\sum_{k \geq 1} e^{-\pi \alpha (k-1/2)^2} \geq 2 e^{- \frac{\pi \alpha}{4}}.
    \end{equation}
    ($\alpha < 1$): We use the Poisson summation formula and the fact from above that $\frac{e^{-3 \pi}}{1 - e^{-\pi}} \leq \frac{1}{11000}$. We start with the upper estimate:
    \begin{align}
        \sqrt{\alpha} \, f_\Z(0;\alpha) & = 1 + 2 \sum_{k \geq 1} e^{-\frac{\pi}{\alpha} k^2} \leq 1 + 2 e^{-\frac{\pi}{\alpha}} \sum_{k \geq 1} e^{-\frac{\pi}{\alpha} (k^2-1)}\\
        & \leq 1 + 2 e^{-\frac{\pi}{\alpha}} \bigg( 1 + \sum_{k \geq 3} e^{-\pi k} \bigg) \leq 1 + \tfrac{22002}{11000} \, e^{-\frac{\pi}{\alpha}}.
    \end{align}
    For the lower estimate we obtain
    \begin{equation}
        \sqrt{\alpha} \, f_\Z(\tfrac{1}{2};\alpha) = 1 + 2 \sum_{k \geq 1} (-1)^k e^{-\frac{\pi}{\alpha} k^2}
        2\geq 1 - 2 e^{-\frac{\pi}{\alpha}} \bigg( 1 + \sum_{k \geq 3} e^{-\pi k} \bigg) \geq 1 - \tfrac{22002}{11000} \, e^{-\frac{\pi}{\alpha}}.
    \end{equation}
\end{proof}
Setting (again) $f(x) = f_\Z(x;\alpha)$, we analyze $f'(x)$ to find solutions for \eqref{eq:critical_3points}. Using $\sqrt{\alpha} f(x) = \theta_3(x;1/\alpha)$, we get a triple product representation for $f$ (cf.~\cite{Mon88});
\begin{equation}
    \sqrt{\alpha} \, f(x) = \prod_{k \geq 1} \left(1 - e^{- 2 k \pi /\alpha} \right) \bigg(1 + 2 \, e^{-(2k-1)\pi/\alpha} \cos(2 \pi x) + e^{-2(2k-1)\pi/\alpha} \bigg).
\end{equation}
For the derivative with respect to $x$ we introduce the function $Q(x;\alpha)$ from \cite{Mon88}:
\begin{align}
    Q(x;\alpha) = 4 \pi \sum_{l \geq 1} & \left(1 - e^{-2 \pi l \alpha} \right) e^{-(2l-1) \pi \alpha}\\
    & \times \prod_{\substack{k \geq 1 \\ k \neq l}} \left(1 - e^{-2 \pi k \alpha}\right)
    \left( 1 + 2 e^{-(2k-1)\pi \alpha} \cos(2 \pi x) + e^{-2(2k-1)\pi \alpha} \right).
\end{align}
The function $f'(x)$ then satisfies
\begin{equation}
    \sqrt{\alpha} \, f'(x) = - \sin(2 \pi x) \, Q(x;\tfrac{1}{\alpha}).
\end{equation}
The function $Q(x;\alpha)$ was analyzed in detail by Montgomery \cite[Lem.~1]{Mon88}: it is 1-periodic, even, positive, strictly decreasing on $(0,1/2)$ and strictly increasing on $(1/2,1)$. In particular, it assumes global maxima in $\Z$ and global minima in $\Z+1/2$. Moreover, we have the following estimates obtained by Montgomery \cite[Lem.~2]{Mon88}.
\begin{lemma}\label{lem:Montgomery}
    We define the following functions for $\alpha > 0$.
    \begin{align}
        A(\alpha) =
        \begin{cases}
            \alpha^{-3/2} e^{-\pi/(4\alpha)}, & \alpha < 1,\\
            \left(1 - \frac{1}{3000} \right) 4 \pi e^{-\pi \alpha}, & 1 \leq \alpha.
        \end{cases}
        \quad
        B(\alpha) =
        \begin{cases}
            \alpha^{-3/2}, & \alpha < 1,\\
            \left(1 + \frac{1}{3000} \right) 4 \pi e^{-\pi \alpha}, & 1 \leq \alpha.
        \end{cases}
    \end{align}
    Then, for all $x$ the following estimates hold true:
    $
        A(\alpha) \leq Q(x;\alpha) \leq B(\alpha).
    $
\end{lemma}
We use the factorization $\sqrt{\alpha} f'(x) = -\sin(2 \pi x) Q(x;1/\alpha)$ to find common points of the curves $c_1$ and $c_2$ (see Figure \ref{fig:curves}) to have solutions to \eqref{eq:critical_3points}:
\begin{equation}
    c_1: \frac{-f(x_2)}{f(2x_2)} = 2 \cos(2 \pi x_1) \frac{Q(2x_1,\tfrac{1}{\alpha})}{Q(x_1,\tfrac{1}{\alpha})}
    \; \text{ and } \;
    c_2: \frac{-f(x_1)}{f(2x_1)} = 2 \cos(2 \pi x_2) \frac{Q(2x_2,\tfrac{1}{\alpha})}{Q(x_2,\tfrac{1}{\alpha})}.
\end{equation}
\begin{figure}[htp]
    \hfill
    \begin{subfigure}[t]{.225\textwidth}
        \includegraphics[width=\textwidth]{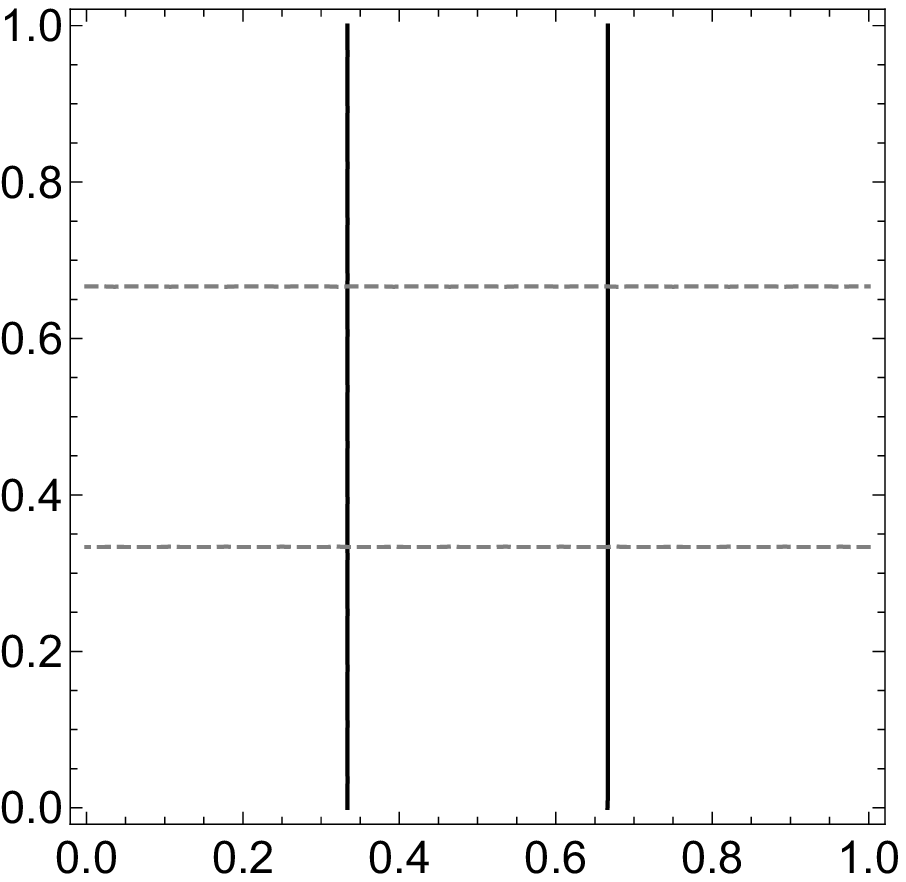}
        \caption{$\alpha = 0.1$}
    \end{subfigure}
    \hfill
    \begin{subfigure}[t]{.225\textwidth}
        \includegraphics[width=\textwidth]{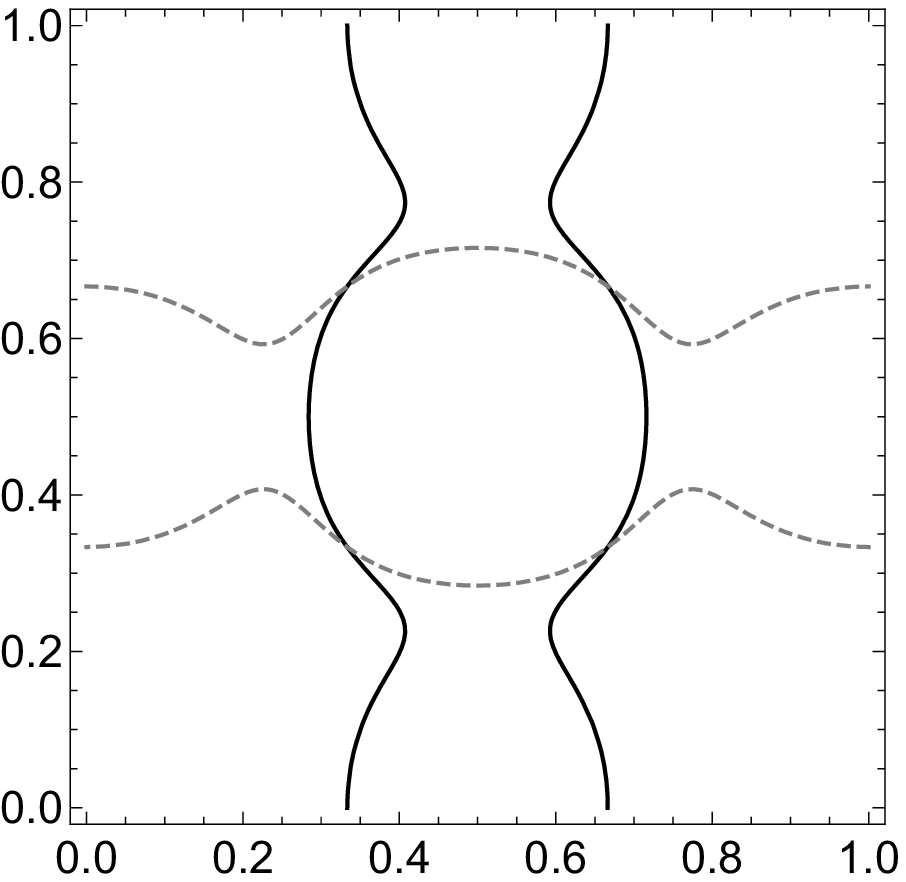}
        \caption{$\alpha = 2.0$}
    \end{subfigure}
    \hfill
    \begin{subfigure}[t]{.225\textwidth}
        \includegraphics[width=\textwidth]{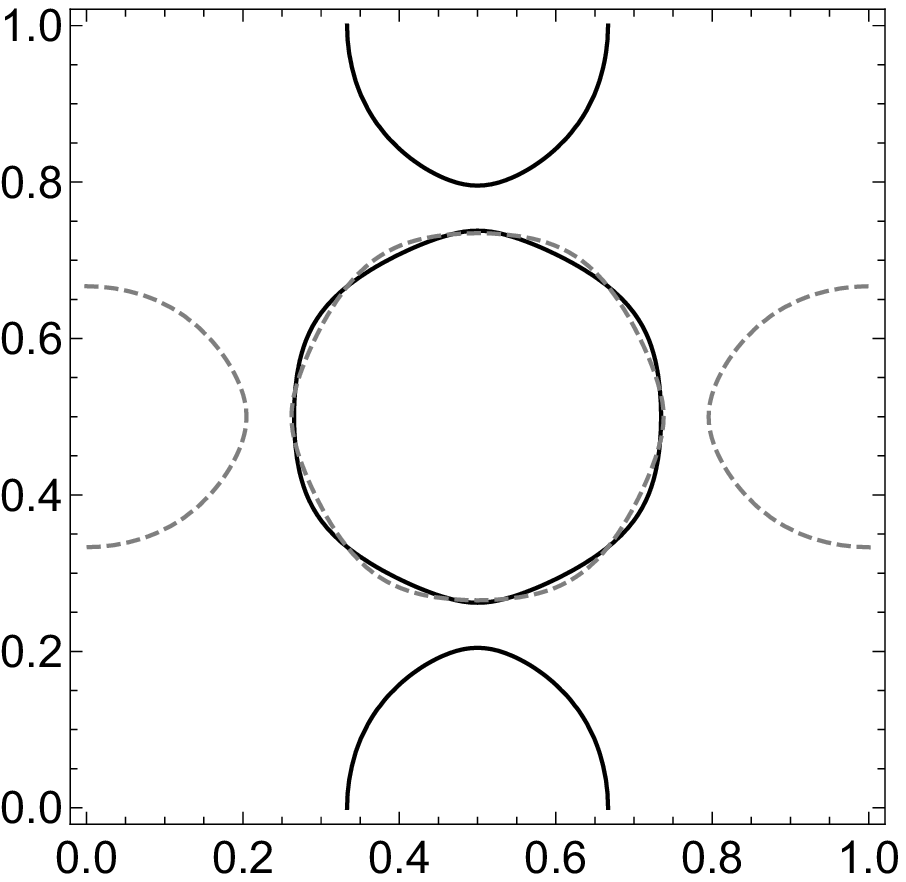}
        \caption{$\alpha = 3.5$}
    \end{subfigure}
    \hfill
    \begin{subfigure}[t]{.225\textwidth}
        \includegraphics[width=\textwidth]{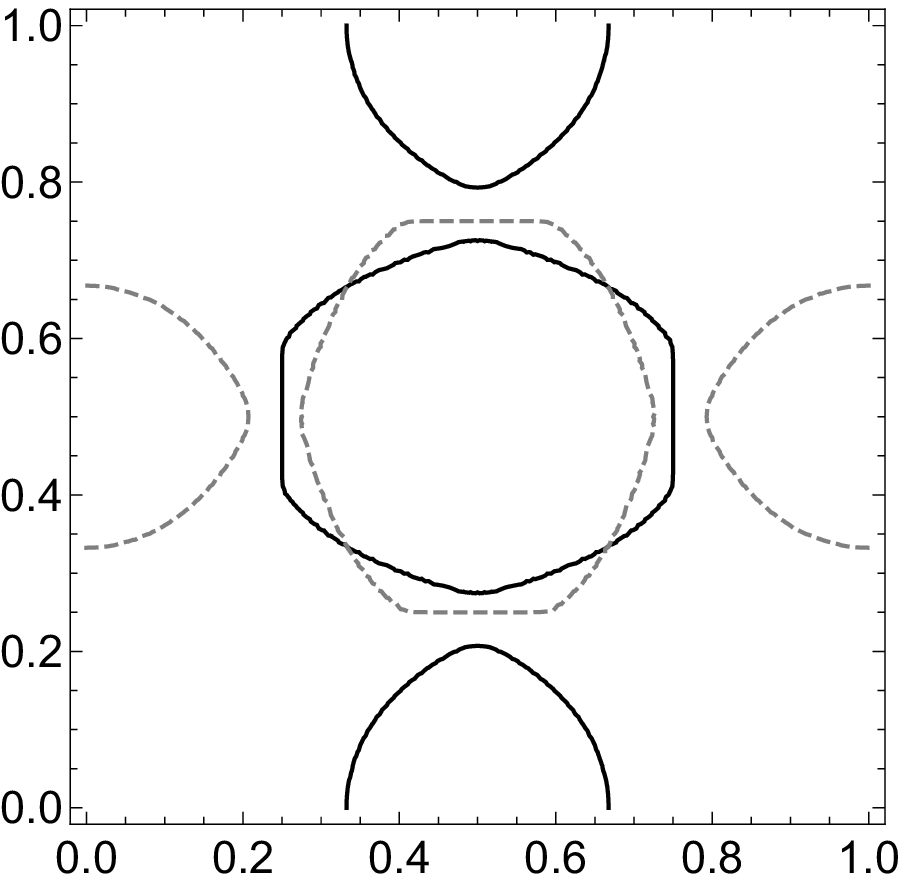}
        \caption{$\alpha = 30.0$}
    \end{subfigure}
    \hspace{\fill}
    
    \caption{$c_1$: \textbf{---}, and $c_2$: \textcolor{gray}{- - -}, in $[0,1]^2$. For most $\alpha$ we find four points of intersection, except for some $\alpha$ around 3.5. Plots created with \textit{Mathematica}~\cite{Mathematica}.}
    \label{fig:curves}
\end{figure}

($\alpha \to 0$): Let us analyze the case when $\alpha$ is asymptotically 0. By Lemma \ref{lem:f_alpha} and Lemma \ref{lem:Montgomery}, we conclude that for all $x \in \R$ we have
\begin{equation}
    - \frac{f(x)}{f(2x)} \asymp -1
    \quad \text{ and } \quad
    2 \cos(2 \pi x) \frac{Q(2x,\tfrac{1}{\alpha})}{Q(x,\tfrac{1}{\alpha})} \asymp 2 \cos(2 \pi x).
\end{equation}
as $\alpha$ tends to 0. Hence, asymptotically we get the following curves:
\begin{equation}
    c_1: -1 \asymp 2 \cos(2 \pi x_1), \, \forall x_2
    \quad \text{ and } \quad
    c_2: -1 \asymp 2 \cos(2 \pi x_2), \, \forall x_1.
\end{equation}
This proves that indeed, as $\alpha$ tends to 0, we only have the solutions $x_0$ and $x_0^*$ or the solutions $\widetilde{x}_0$ and $\widetilde{x}_0^*$ given by \eqref{eq:solutions}.

($\alpha \to \infty$): For $\alpha$ going to infinity we are not able to describe the exact shape of the limiting curves $c_1$ and $c_2$ (which are rotations of one another) in the unit square. The numerics tell us that $c_1$ and $c_2$ consist of two closed curves (seen periodically). Also, for any $\alpha$ the points $(\frac{1}{3},0)$, $(\frac{2}{3},0)$ are on $c_1$ and $(0,\frac{1}{3})$, $(0,\frac{2}{3})$ are part of $c_2$.

\subsubsection{Comparison}
Note that $\Gamma_{x_0,x_0^*}$ is a lattice and so by Theorem \ref{thm:Montgomery} it has lower energy than the hexagonal lattice $\Lambda_2$. Scaling $\Gamma_{x_0,x_0^*}$ to unit density, we may write
\begin{equation}
    \Gamma_{x_0,x_0^*} = \sqrt{3}
    \begin{pmatrix*}[r]
        \tfrac{1}{3} & -\tfrac{1}{3}\\[.6ex]
        \tfrac{1}{3} & \tfrac{2}{3}
    \end{pmatrix*}
    \Z^2 \cong \sqrt{\tfrac{2}{3}}
    \begin{pmatrix}
        1 & \tfrac{1}{2}\\[.6ex]
        0 & \tfrac{3}{2}
    \end{pmatrix} \Z^2,
\end{equation}
where the second version is a clockwise rotation of $\Gamma_{x_0,x_0^*} $ by 45 degrees. A question which arises is whether the square lattice $\Z^2$ may also have lower energy than $\Gamma_{x_0,x_0^*}$.

\begin{theorem}
    For all $\alpha>0$, $\Z^2$ has lower energy than the configuration $\Gamma_{x_0,x_0^*}$. 
\end{theorem}

\begin{proof}
The energy of $\Gamma_{x_0,x_0^*}$ can be estimated from below by 
\begin{equation}
	E_{\phi_{\pi \alpha}}(\Gamma_{x_0,x_0^*})
    = \sum_{k,l \in \Z} e^{-\frac{2 \pi \alpha}{3} \left(k^2 + k l + \frac{5}{2} l^2\right)} - 1
    \geq 2 \, e^{-\frac{2 \pi \alpha}{3}}.    
\end{equation}
By Lemma \ref{lem:f_alpha}, for $\alpha\geq 1$, it suffices to show
\begin{equation}
    E_{\phi_{\pi\alpha}}(\Z^2) = f_{\Z}(0,\alpha)^2-1 \leq \left(1 + \tfrac{22002}{11000} e^{-\pi \alpha}\right)^2-1 < 2 \, e^{-\frac{2 \pi \alpha}{3}}.
\end{equation}
By the change of variable $x=e^{-\frac{\pi \alpha}{3}}$, this reduces to showing for all $x\in(0,e^{-\frac{\pi}{3}}]$
\begin{align}
    0> \left(1 + \tfrac{22002}{11000} \, x^3\right)^2-1-2 x^2 = \tfrac{22002^2}{11000^2} \, x^2\left(x^4 +2\cdot \tfrac{11000}{22002}x - 2\cdot \tfrac{11000^2}{22002^2}\right).
\end{align}
Note that the quartic function in brackets is strictly monotonically increasing on the interval $(0,e^{-\frac{\pi}{3}}]\subseteq (0,0.4)$, so that we can estimate as follows:
\begin{equation}
x^4 +2\cdot \tfrac{11000}{22002}x - 2\cdot \tfrac{11000^2}{22002^2} \leq 0.4^4 +2\cdot \tfrac{11000}{22002}\cdot 0.4 - 2\cdot \tfrac{11000^2}{22002^2}< 0.4^4+0.4-0.5<0.
\end{equation}
This proves the claim for $\alpha\geq 1$. Since all lattices of unit density in $\R^2$ are symplectic and as 2-dimensional Gaussians are eigenfunctions of the symplectic Fourier transform (cf.\ \cite[Sec.~2.2]{Fau18}), the case $\alpha<1$ follows by the symplectic Poisson summation formula. Combining the cases then gives the full result for all $\alpha > 0$.
\end{proof}

\bibliographystyle{plain}


\end{document}